\documentclass[a4paper, 11pt, oneside]{amsart}

\usepackage[top=25mm,bottom=25mm,left=20mm,right=20mm]{geometry}

\usepackage{mathrsfs}
\usepackage{amsfonts}
\usepackage{amssymb}
\usepackage{amsxtra}
\usepackage{color}
\usepackage{dsfont}
\usepackage[english,polish]{babel}
\usepackage[colorlinks,citecolor=blue,urlcolor=blue,bookmarks=true]{hyperref}
\hypersetup{
pdfpagemode=UseNone,
pdfstartview=FitH,
pdfdisplaydoctitle=true,
pdfborder={0 0 0}, % No link borders
pdftitle={Asymptotic behavior of densities of unimodal convolution semigroups},
pdfauthor={Wojciech Cygan, Tomasz Grzywny and Bartosz Trojan},
pdflang=en-US
}

\newcommand{\sprod}[2]{\langle {#1}, {#2}\rangle}
\newcommand{\norm}[1]{{\lvert #1 \rvert}}
\newcommand{\abs}[1]{{\lvert {#1} \rvert}}
\renewcommand{\atop}[2]{\substack{{#1}\\{#2}}}

\newcommand{\ind}[1]{{\mathds{1}_{{#1}}}}

\newcommand{\CC}{\mathbb{C}}
\newcommand{\ZZ}{\mathbb{Z}}

\newcommand{\RR}{\mathbb{R}}
\newcommand{\PP}{\mathbb{P}}

\newcommand{\Ss}{\mathbb{S}}

\newcommand{\calR}{\mathcal{R}}
\newcommand{\calL}{\mathcal{L}}
\newcommand{\calS}{\mathcal{S}}
\newcommand{\calF}{\mathcal{F}}

\newcommand{\calM}{\mathcal{M}}
\newcommand{\calB}{\mathcal{B}}

\newcommand{\RInf}{\calR^\infty_{\alpha} }
\newcommand{\ROrg}{\calR^0_{\alpha} }

\newcommand{\pl}[1]{\foreignlanguage{polish}{#1}}

\newtheorem{theorem}{Theorem}
\newtheorem{proposition}{Proposition}[section]
\newtheorem{lemma}{Lemma}
\newtheorem{corollary}{Corollary}

\newtheorem{example}{Example}
\newtheorem{remark}{Remark}

\title{Asymptotic behavior of densities \\ of unimodal convolution semigroups}

\author{Wojciech Cygan}
 \thanks{Research of W.~Cygan was supported by National Science Centre (Poland), Grant DEC-2013/11/N/ST1/03605}
\author{Tomasz Grzywny}

\author{Bartosz Trojan}

\address{
		Wojciech Cygan\\
		Instytut Matematyczny\\
		Uniwersytet \pl{Wroc{\lll}awski}\\
		Pl. Grun\-waldzki 2/4\\
		50-384 \pl{Wroc{\lll}aw}\\
		Poland}
\email{cygan@math.uni.wroc.pl}

\address{
	Tomasz Grzywny\\
	 \pl{Wydzia{\lll}} Matematyki\\
	Politechnika \pl{Wroc{\lll}awska}\\
	Wyb. \pl{Wyspia\'{n}skiego} 27\\
	50-370 \pl{Wroc\l{}aw}\\
	Poland
}
\email{tomasz.grzywny@pwr.edu.pl}

\address{
		Bartosz Trojan\\
		Instytut Matematyczny\\
		Uniwersytet \pl{Wroc{\lll}awski}\\
		Pl. Grun\-waldzki 2/4\\
		50-384 \pl{Wroc{\lll}aw}\\
		Poland}
\email{Bartosz.Trojan@math.uni.wroc.pl}


\begin{document}
\selectlanguage{english}

\begin{abstract}
We prove the asymptotic formulas for the densities of i\-so\-tro\-pic unimodal
convolution semigroups of probability measures on $\RR^d$ under the assumption that
its L\'{e}vy--Khintchine exponent is regularly varying of index between $0$ and $2$. 
\end{abstract}

\maketitle

\section{Introduction}
Studying the asymptotic behavior of the densities of convolution semigroups of probability measures 
has attracted much attention of experts from mathematics and statistics for many years.
The main goal of this paper is to prove the asymptotic formulas for the densities $p(t,x)$ of \textit{isotropic unimodal}
convolution semigroups of probability measures on $\RR ^d$ assuming that the corresponding L\'{e}vy--Khintchine exponent
$\psi (\xi)$ is regularly varying of index $\alpha$ between 0 and 2.

One of the first results of this kind is due to P\'{o}lya (1923, $d=1$) and to Blumenthal and Getoor (1960, $d>1$)
and provides the asymptotic behavior of the transition density $p_\alpha (t,x)$ of the isotropic $\alpha$-stable
process $\mathbf X_\alpha$ in $\RR ^d$, $\alpha \in (0, 2)$. The remarkable scaling property of the process
$\mathbf X_\alpha$, namely
\[
	p_\alpha (t,x) = 
	t^{-d/\alpha}
	p_\alpha \big(1,t^{-1/\alpha }x \big),
\]
implies that the following \textit{strong ratio limit property} holds
\begin{align}
	\label{SRL_Stable}
	\lim _{t\norm{x}^{-\alpha} \to +\infty}
	\frac{p_\alpha(t,x) }{p_\alpha (t,0)} 
	=
	1,
\end{align}
and that $p_{\alpha }(t,0) = p_{\alpha }(1,0)t^{-d/\alpha}$. Concerning the ratio limit property in general setting
we refer to the articles \cite{MR0132600}, \cite{MR0174089} (irreducible and aperiodic Markov chains) and to the articles
\cite{Stone1}, \cite{Stone2} and \cite{MR2139571} (continuous time processes, in particular L\'{e}vy processes).

Let us recall that the result of P\'{o}lya ($d=1$) and of Blumenthal and Getoor ($d>1$) reads as follows
\begin{align}
	\label{eq:5}
	\lim _{\norm{x} \to +\infty}
	\norm{x}^{d+\alpha}
	p_\alpha(1,x) 
	=
	\mathcal{A}_{d, \alpha}
\end{align}
where
\begin{align}
	\label{eq:8}
	\mathcal{A}_{d, \alpha}
	=
	\alpha 2^{\alpha -1} 
	\pi^{-d/2-1}
	\sin \left(\frac{\alpha \, \pi}{2}\right)
	\Gamma \left(\frac{\alpha}{2}\right)
	\Gamma \left(\frac{\alpha+ d}{2}\right),
\end{align}
equivalently
\[
	\lim_{t \norm{x}^{-\alpha} \to 0} \frac{p_\alpha(t, x)}{t \norm{x}^{-d - \alpha}} = \mathcal{A}_{d, \alpha}.
\]
It follows that \footnote{$A \asymp B$ means that $c B \leq A \leq C B$, for some constants $c,C>0$.}
\[
	p_{\alpha }(t,x) \asymp \min \left\{ p_{\alpha} (t,0), t |x|^{-d-\alpha} \right\}
\]
uniformly in $t>0$ and $x\in \RR ^d$. Observe that for the Cauchy semigroup ($\alpha = 1$) we have the equality
\[
	p_1(t,x) = {\mathcal A}_{d,1} t \left(t^2+|x|^2\right)^{-\frac{d+\alpha}{2}}.
\]

The proofs of statement \eqref{eq:5} given by P\'{o}lya and by Blumenthal and Getoor were based on Fourier analytic
techniques. In \cite{bendikov} Bendikov used a different approach utilizing the Bochner's method of subordination.
For completeness we would like to list other related papers \cite{MR2552308}, \cite{MR1301283}, \cite{MR2827465} and
\cite{MR941977} and \cite{MR1884159}. We also mention that with the aid of subelliptic estimates developed
by G\l{}owacki \cite{glo}, analogous asymptotic formulas were obtained by Dziuba\'nski \cite{MR1085177} for strictly
stable semigroups of measures in the case of homogeneous Lie groups, in particular $\RR^d$.

Let $\mathbf X $ be an isotropic L\'{e}vy process having transition density $p(t,x)$. Then
\begin{align*}
	\int _{\RR ^d} e^{i\sprod{\xi}{x}} p(t, x) {\: \rm d} x
	=
	e^{-t\psi (\xi)}
\end{align*} 
where the L\'{e}vy--Khintchine exponent $\psi$ is a radial function. We also assume that the L\'{e}vy measure has
\textit{unimodal} density. As it was recently proved in the paper \cite{MR3165234}, under weak scaling
properties of $\psi$ at infinity
\begin{align}
	\label{BGR_bounds}
	p(t,x)\asymp \min \left\{ p(t,0), t \norm{x}^{-d} \psi \big(|x|^{-1} \big) \right\}
\end{align}
locally in space and time variables. Hence, if the characteristic exponent varies regularly at infinity one can use
\eqref{BGR_bounds} to obtain estimates for $p(t,x)$. The results of the present article implies the estimates
for large $t$ and $\norm{x}$.

The estimates \eqref{BGR_bounds} are also valid for the case when $\psi$ varies regularly at zero. Here, additionally, we 
need to assume that $e^{-t_0 \psi}$ is integrable on $\RR^d$ for some $t_0 > 0$. Indeed, as a consequence of Theorem
\ref{densityAsymp}, monotonicity of $p(t,\cdot)$ and Remark \ref{Rem_p(t,0)}
(compare the proof of \cite[Theorm 21]{MR3165234}) one can find $M > 0$ such that \eqref{BGR_bounds} holds for
all $\norm{x}, t > M$.

In this article, for $\psi$ varying regularly of index $\alpha \in (0,2)$ at zero, we prove that
\begin{align}
	\label{eq:6}
	\lim_{\atop{\norm{x} \to +\infty}{t \psi(\norm{x}^{-1}) \to 0}}
	\frac{p(t, x)}{t \norm{x}^{-d} \psi\big(\norm{x}^{-1}\big)} = 
	\mathcal{A}_{d,\alpha}
\end{align}
where the constant $\mathcal{A}_{d,\alpha}$ is given by the formula \eqref{eq:8} (see Theorem \ref{densityAsymp}).

The main step in getting the asymptotic \eqref{eq:6} is to find the asymptotic behavior of tails
$\PP\big( \norm{X_t}\geq r \big)$. We prove that for some positive constant $\mathcal{C}_{d, \alpha}$
\begin{align*}
	\lim_{\atop{r \to +\infty}{t \psi(r^{-1}) \to 0}}
	\frac{\PP\big(\norm{X_t} \geq r\big)} {t \psi(r^{-1})} = \mathcal{C}_{d, \alpha}.
\end{align*}
This asymptotic expression may be viewed as an uniform in time variant of the Tauberian theorem. In the proof,
we interpret the asymptotic of the Laplace transform of $\PP\big(\norm{X_t} \geq \sqrt{r} \big)$ as a prescribed
distributional limit for a dense class of test functions. Then using equicontinouity of the distributions we conclude
the convergence. The above technique was previously used in \cite{bct} where subordinated random walks on $\ZZ^d$ were
considered. 

Beside the transition density $p(t,x)$ we also investigate the L\'{e}vy measure $\nu$ of the process $\mathbf X$.
It is rather usual that $\nu$ bears an asymptotic resemblance to $p(t,x )$ since in a vague sense
$\nu(x) = \lim _{t \to 0^+} t^{-1} p(t, x)$. Therefore, from \eqref{eq:6} we may deduce the following
asymptotic behavior of $\nu$ (see Theorem \ref{LevyDensity2})
\begin{align}
	\label{eq:7}
	\lim_{\norm{x} \to +\infty}
	\frac{\nu(x)}{\norm{x}^{-d} \psi\big(\norm{x}^{-1}\big)} =\mathcal{ A}_{d, \alpha}.
\end{align}
The fact that the function $\psi$ appears both in asymptotic formulas for $\nu$ and $p(t, x)$ is natural from
the point of view of the pseudo-differential calculus and the spectral theory. Moreover, as $\psi$ being connected with
the Fourier transform of $p(t,x)$, the asymptotic behavior of $\psi$ at infinity translates into
the asymptotic formulas for $p(t,x)$ and $\nu$ at zero.

A natural question arises:
\begin{quote}
	\em
	Assume that the density $\nu$ of the L\'{e}vy measure of the process $\mathbf X$ with
	the L\'{e}vy--Khintchine exponent $\psi$ has the asymptotic behavior of the form \eqref{eq:7} at infinity.
	Is it true that $\psi $ has a certain prescribed behavior at zero?
\end{quote}
Surprisingly, the answer is affirmative. It turns out that $\psi$ is necessarily regularly varying 
(see Theorem \ref{LevyDensity2}). This interesting observation is a consequence of an application
of the famous Drasin--Shea theorem \cite[Theorem 6.2]{shea}. In the one dimensional case this result was partially obtained in
the unpublished manuscript by Bendikov \cite{Ben_unpublished}. 

In Subsection \ref{sec:4.2} we investigate the asymptotic of the potential measure provided $d \geq 3$
(see Theorem \ref{GreenAsymp}) which allows us to prove that for some positive constant $\mathcal{\tilde{A}}_{d,\alpha}$
\[
	\lim_{\norm{x} \to +\infty} \norm{x}^d \psi\big(\norm{x}^{-1}\big) G(x) = \mathcal{\tilde{A}}_{d,\alpha}
\]
where $G(x)$ is a part of the density of the potential measure absolutely continuous with respect to Lebesgue
measure (see Corollary \ref{cor:1}). Such limit was proved in \cite[Theorem 3.1 and Theorem 3.3]{rsv} for a 
class subordinate Brownian motions governed by complete Bernstein functions using the Tauberian theorem for
the potential measure of subordinator.

We also prove the corresponding theorems with zero replaced by infinity and vice versa (see Theorem
\ref{Tail2}, Theorem \ref{densityAsymp1}, Theorem \ref{LevyDensity} and Theorem \ref{GreenAsymp2}).

A vast class of examples of isotropic unimodal L\'{e}vy processes constitute subordinate Brownian
motions. For such processes we have $\psi (x) = \phi \big(|x|^2\big)$ where $\phi $ is the Laplace
exponent of the corresponding subordinator. In particular, $\phi$ is a Bernstein function. For instance
the monograph \cite{ssv} gives many cases and classes of Bernstein functions in
its closing list of examples. Moreover, \cite[Theorem 2.5]{bsc} shows that for a given function $f$ such that
$f(x)$ and $xf'(x)$ are regularly varying of index $\beta\in[0,1)$ at zero, one can find a complete Bernstein function
$\phi$ such that 
\[
	\lim_{x \to 0^+} \frac{\phi(x)}{f(x)} = 1.
\] 
Our results in turn imply the asymptotic formulas for the L\'{e}vy measure and the transition density
of the corresponding subordinate Brownian motion. Let us observe that the symmetric $\alpha$-stable processes
are subordinate Brownian motions since we may take $\phi (\lambda )=\lambda ^{\alpha}$, $\alpha  \in (0, 1)$.
Thus our theorems may be regarded as significant extensions of the classical result \eqref{eq:5}.

In the present article we covered the case when the L\'{e}vy--Khintchine exponent $\psi$ is $\alpha$-regular for
$\alpha \in (0, 2)$. The corresponding results for the case $\alpha = 0$ is the subject of the forthcoming paper 
\cite{grt}. The case $\alpha = 2$ is the ongoing project.

\section{Preliminaries}
Let $\mathbf{X}=(X_t : t \geq 0)$ be an isotropic L\'{e}vy process in $\RR ^d$, i.e. $\mathbf{X}$ is a c\'{a}dl\'{a}g
stochastic process with a distribution denoted by $\PP$ such that $X_0=0$ almost surely, the increments of $\mathbf{X}$
are independent with a radial distribution $p(t, \,\cdot\,)$ on $\RR ^d\setminus \{0\}$. This is equivalent with radiality
of the L\'{e}vy measure and the L\'{e}vy--Khintchine exponent. In particular, the characteristic
function of $\mathbf{X}$ has a form
\begin{align}
	\label{eq:9}
	\int_{\RR^d} e^{i \sprod{\xi}{x}} p(t, {\rm d} x) = e^{-t\psi (\xi)}
\end{align} 
where
\begin{equation}
	\label{eq:28}
	\psi (\xi) = \int_{\RR ^d} \big(1-\cos\sprod{\xi}{x}\big) \nu ({\rm d}x)
	+ \eta \norm{\xi}^2,
\end{equation}
for some $\eta \geq 0$. We are going to abuse the notation by setting $\psi(r)$ for $r > 0$ to be equal to $\psi(\xi)$ for any $\xi \in \RR^d$
with $\norm{\xi} = r$. Since the function $\psi$ is not necessary monotonic, it is conveniently to work with $\psi^*$
defined by 
\begin{equation*}
	\psi^*(u) = \sup_{s \in [0, u]} \psi(s)
\end{equation*}
for $u \geq 0$. Let us recall that for $r, u \geq 0$ (see \cite[Theorem 2.7]{WHoh})
\begin{equation}
	\label{eq:2}
	\psi (ru)
	\leq 
	\psi^*(ru)
	\leq 2(r^2+1)
	\psi^*(u).
\end{equation}

A Borel measure $\mu$ is isotropic \emph{unimodal} if it is absolutely continuous on $\RR^d \setminus \{0\}$
with a radial and radially non-increasing density. A L\'{e}vy process $\mathbf X$ is isotropic unimodal if
$p(t, \: \cdot \:)$ is isotropic unimodal for each $t > 0$. In Section 4 we consider a subclass of isotropic processes
consisting of isotropic unimodal L\'{e}vy processes. They were characterized by Watanabe in \cite{Watanabe} as those
having the isotropic unimodal L\'{e}vy measure. A remarkable property of these processes is (see \cite[Proposition 2]{MR3165234})
\begin{equation}
	\label{eq:27}
	\psi^*(u) \leq \pi^2 \psi(u)
\end{equation}
for all $u \geq 0$.

\subsection{Regular variation}
A function $\ell: [x_0, +\infty) \rightarrow (0, +\infty)$, for some $x_0 > 0$, is called
\emph{slowly varying at infinity} if for each $\lambda > 0$
\[	
	\lim_{x \to +\infty} \frac{\ell(\lambda x)}{\ell(x)} = 1.
\]
We say that $f: [x_0, +\infty) \rightarrow (0, +\infty)$ is \emph{regularly varying of index
$\alpha \in \RR$ at infinity}, if $f(x) x^{-\alpha}$ is slowly varying at infinity. The set of
regularly varying functions of index $\alpha$ at infinity is denoted by $\RInf$. In particular, if $f \in \RInf$
then
\[
	\lim_{x \to +\infty}
	\frac{f(\lambda x)}{f(x)}
	=\lambda^\alpha,\quad \lambda>0.
\]
A function $f$ is regularly varying of index $\alpha \in \RR$ at \emph{zero} if $x \mapsto f\big(x^{-1}\big)^{-1}$
belongs to $\RInf$. The set of regularly varying functions of index $\alpha$ at zero is denoted by $\ROrg$. 
The following property of a slowly varying function at zero appears to be very useful (see \cite{pot}, see also
\cite[Theorem 1.5.6]{bgt}). For every $C > 1$ and $\epsilon > 0$ there is $0 < \delta \leq x_0$ such that for all
$0 < x, y \leq \delta$
\begin{equation}
	\label{eq:14}
	\ell(x) \leq C \ell(y) \max\{x/y, y/x\}^\epsilon.
\end{equation}

\section{Strong Ratio Limit Theorem}
%Introductory text:  SRLT on compacts, BGR bound and scaling, our tasks
In this section we discuss the strong ratio limit property. Let $\mathbf{X}$ be an isotropic L\'{e}vy process
in $\RR ^d$ with the L\'{e}vy--Khintchine exponent $\psi$ satisfying
\begin{equation}
	\label{eq:3}
	e^{-t_0 \psi} \in L^1\big(\RR ^d\big)
\end{equation}
for some $t_0 > 0$. It is well-known that then $\mathbf{X}$ has a continuous transition density $p(t,x)$ such that
for every compact subset $K \subset \RR^d$
\begin{equation}
	\label{eq:slr}
	\lim _{t\to \infty}\frac{p(t,x)}{p(t,0)}=1
\end{equation}
uniformly with respect to $x \in K$. Our aim is to extend this property to the non-compact case. 

\begin{theorem}
	\label{regular_SRLT}
	If $\psi \in \ROrg$ for some $\alpha \in (0, 2]$ and $e^{-t_0 \psi}\in L^1\big(\RR ^d\big)$, for some 
	$t_0 > 0$, then
	\begin{align*}
		\lim _{\atop{t\to +\infty}{t\psi (\norm{x}^{-1})\to +\infty}} 
		\frac{p(t,x)}{p(t,0)}=1.
	\end{align*}
\end{theorem}
\begin{proof}
	First, we show that there is $C > 0$ such that for $t$ large enough
	\begin{equation}
		\label{eq:21}
		\int _{\RR ^d}
		e^{-t\psi (\xi)} \norm{\xi}  {\: \rm d}\xi 
		\leq
		C
		\left(\psi ^{-}(1/t)\right)^{d+1}
	\end{equation}
	where $\psi^{-}(u) = \min\{r \geq 0 : \psi(r) \geq u\}$. We observe that integrability of $e^{-2 t_0 \psi}$ on $\RR^d$ implies that  $e^{-2 t_0 \psi} < 1$. Moreover, by Riemann--Lebesgue lemma
	$e^{-2 t_0 \psi(\xi)}$ tends to zero if $\norm{\xi}$ approaches infinity. Hence, for any $\delta>0$ there is $\gamma > 0$ such that
	for all $\norm{\xi} \geq \delta$ 
	\[
		e^{-2 t_0 \psi(\xi)} \leq 1 - \gamma.
	\]
	Therefore, for $t > 4 t_0$
	\begin{equation}
		\label{eq:29}
		\int_{\norm{\xi} \geq \delta} e^{-t \psi(\xi)} \norm{\xi} {\: \rm d}\xi
		\leq
		(1 - \gamma)^{t/(4t_0)} \int_{\norm{\xi} \geq \delta} e^{-2 t_0 \psi(\xi)} \norm{\xi} {\: \rm d} \xi.
	\end{equation}
	By \eqref{eq:14}, there is $\delta > 0$ such that
	for all $0 < \norm{x}, \norm{y} \leq \delta$,
	\begin{equation}
		\label{eq:17}
		\psi(x) \leq 2 \psi(y) (\norm{x}/\norm{y})^\alpha \max\{\norm{x}/\norm{y}, \norm{y}/\norm{x}\}^\epsilon.
	\end{equation}
	 Since $ \psi(\psi^{-}(1/t)) =1/t$ and
	for $t > t_0$ sufficiently large $\psi^{-}(1/t) < \delta$, we get
	\[
		2t \psi(\xi) \geq \big(\norm{\xi} \psi^{-}(1/t)^{-1} \big)^{\alpha}
		\min\left\{ \norm{\xi} \psi^{-}(1/t)^{-1}, \norm{\xi}^{-1} \psi^{-}(1/t)\right\}^{\epsilon}.
	\]
	Hence, by the change of variables $u=\xi \psi^{-}(1/t)^{-1}$ we obtain	\footnote{We write $A \lesssim B$ if there an absolute constant $C > 0$ such that $A \leq C B$.}
	\begin{equation}
		\label{eq:30}
		\int_{\norm{\xi} \leq \delta}
		e^{-t\psi(\xi)} \norm{\xi} {\: \rm d}\xi
		\lesssim
		\big(\psi^{-}(1/t)\big)^{d+1}.
	\end{equation}
	By putting \eqref{eq:30} together with \eqref{eq:29} we arrive at \eqref{eq:21}.

	Now, by the Fourier inversion formula for $x, y \in \RR^d$
	\[
		\abs{D_x p(t, y)}
		\lesssim
		\int_{\RR^d} 
		e^{-t\psi(\xi)} \abs{\sprod{\xi}{x}} {\: \rm d}\xi.
	\]
	Hence, by the mean value theorem and the estimate \eqref{eq:21} we get
	\begin{equation*}
		\abs{p(t, 0) - p(t, x)} 
		\leq
		\sup_{0 \leq \theta \leq 1}
		\abs{D_x p(t, \theta x)}
		\lesssim
		\norm{x} \big(\psi^{-}(1/t)\big)^{d+1}.
	\end{equation*}
	Because	
	\begin{equation}
		\label{eq:32}
		p(t, 0) \gtrsim \int_{\norm{\xi} \leq \psi^{-}(1/t)} e^{-t \psi(\xi)} {\: \rm d}\xi
		\gtrsim
		\big(\psi^{-}(1/t)\big)^d,
	\end{equation}
	we obtain
	\begin{equation}
		\label{eq:31}
		\abs{p(t, 0) - p(t, x)} \leq C p(t, 0) \norm{x} \psi^{-}(1/t).
	\end{equation}
	Therefore, it is enough to show
	\begin{equation}
		\label{eq:16}
		\lim_{t \psi(\norm{x}^{-1}) \to +\infty} \norm{x} \psi^{-}(1/t) = 0.
	\end{equation}
	By \eqref{eq:2}, for $\lambda \geq 1$ and $u \geq 0$ we have
	\[
		\psi^*(\lambda u) \leq 4 \lambda^2 \psi^*(u),
	\]
	thus
	\[
		\psi^{-} (\lambda u) \geq \frac12 \sqrt{\lambda} \psi^{-}(u).
	\]
	By taking $\lambda = t \psi\big(\norm{x}^{-1}\big)$ and $u = 1/t$ we obtain
	\begin{equation}
		\label{eq:34}
		\norm{x} \psi^{-}(1/t) 
		\leq \frac12
		\big(t \psi\big(\norm{x}^{-1}\big)\big)^{-1/2},
	\end{equation}
	and the proof is finished.
\end{proof}
\begin{remark}
	Instead of regular variation at zero one may assume that there are $\alpha \in (0, 2]$ and $c>0$ such that for all
	$\lambda \geq 1$, $x \in \RR^d$
	\begin{equation}
		\label{eq:15}
		\psi (\lambda x)\geq c \, \lambda ^{\alpha}\psi (x).
	\end{equation}
	Then for each $t > 0$, $e^{-t \psi}$ is integrable on $\RR^d$, (see \cite[Lemma 7]{MR3165234}).
	Hence, a measure $p(t, {\rm d} x)$ has a smooth and integrable density $p(t, x)$ (see \cite[Theorem 1]{Schilling}).
	In particular, we have \eqref{eq:21}, what again implies \eqref{eq:31}. Using \eqref{eq:16} we conclude
	\begin{align*}
		\lim _{t\psi (\vert x\vert ^{-1})\to \infty}\frac{p(t,x)}{p(t,0)}=1.
	\end{align*}
\end{remark}

For the case when the L\'{e}vy--Khintchine exponent $\psi$ varies regularly at infinity we have the following.
\begin{proposition}
	If $\psi \in \RInf$ for some $\alpha \in (0, 2]$ then
	\begin{align*}
		\lim _{\atop{\norm{x}\to 0}{t\psi (\norm{x}^{-1})\to +\infty}} \frac{p(t,x)}{p(t,0)}=1.
	\end{align*}
\end{proposition}
\begin{proof}
	By \ref{eq:14} we have that for any $\underline{x}>0$ there is $c>0$, if $x \geq \underline{x}$ and $\lambda \geq 1$ then
	\[
		\psi (\lambda x)\geq c\, \lambda ^{\alpha/2} \psi (x).
	\]
	Whence, by \cite[Lemma 16]{MR3165234}, for any $T > 0$ there is $C_T > 0$ such that for all $t \in (0, T)$
	\begin{align*}
		\int _{\RR ^d}
		e^{-t\psi (\xi)}
		\norm{\xi} {\: \rm d} \xi
		\leq
		C_T
		\left(\psi^{-}(1/t)\right)^{d+1}.
	\end{align*}
	Again, by the mean value theorem and \eqref{eq:32} we get
	\begin{equation}
		\label{eq:35}
		\abs{p(t,0) - p(t,x)} 
		\lesssim 
		\norm{x} p(t, 0) \psi^{-} (1/t),
	\end{equation}
	for $t \in (0, T)$.
	
	Let $\epsilon > 0$. By \eqref{eq:slr}, there is $T > 0$ such that for all $t \geq T$ and $\norm{x} \leq \epsilon$
	\[
		\abs{p(t, x) - p(t, 0)} < \epsilon p(t, 0).
	\]
	Therefore, by \eqref{eq:34} and \eqref{eq:35}, it is enough to take 
	$t \psi\big(\norm{x}^{-1}\big) \geq C_T^2 \epsilon^{-2}$.
\end{proof}

\begin{remark}
	\label{Rem_p(t,0)}
	Writing the Fourier inversion formula in polar coordinates we obtain
	\[
		p(t,0) = \frac{2^{1-d}}{\pi ^{d/2}\Gamma (d/2)}\int _0^\infty e^{-t\psi (r)}r^{d-1} {\: \rm d} r.
	\]
	Hence, after the change of the variables we get
	\begin{align*}
		p(t,0) 
		= 
		\frac{2^{1-d}}{d\pi ^{d/2}\Gamma (d/2)}
		\int _0^\infty e^{-tr}{\: \rm d} \left\{\left(\psi ^-(r)\right)^d\right\}.
	\end{align*}
	Next, applying Karamata theory \cite[Theorem 1.7.1 and 1.7.1']{bgt} we can calculate
	\begin{align*}
		\psi \in \ROrg 
		\Longrightarrow 
		\lim _{t \to +\infty } 
		\frac{p(t,0)}{ \left(\psi ^-(1/t)\right)^d} = 
		\frac{2^{1-d}\Gamma (1+d/\alpha)}{d \pi ^{d/2}\Gamma (d/2)}
	\end{align*}
	and 
	\begin{align*}
		\psi \in \RInf \Longrightarrow \lim _{t\to 0^+ }\frac{p(t,0)}{ \left(\psi ^-(1/t)\right)^d} = 
		\frac{2^{1-d}\Gamma (1+d/\alpha)}{d \pi ^{d/2}\Gamma (d/2)}.
	\end{align*}
\end{remark}
\begin{remark}
	Let us mention that in the case when $t\psi\big(\norm{x}^{-1}\big)$ tends neither to zero nor to infinity,
	the strong ratio limit property may fail. To illustrate this it is enough to consider the Cauchy semigroup
	with the transition density $p_1(t,x)$, cf. \eqref{eq:5}, for which we have
	\begin{equation}
		\label{eq:33}
		\frac{p_1(t,x)}{p_1(t,0)} = \Big( 1+\frac{\norm{x}^2}{t^2}\Big)^{-\frac{d+1}{2}}.
	\end{equation}
	Now, if $t \norm{x}^{-1}$ goes to $\eta$, $\eta \in (0, +\infty)$, then the limit of the ratio \eqref{eq:33}
	depends on $\eta$.
\end{remark}

\section{Asymptotic behavior of the tails}
Suppose that $\mathbf{X}=(X_t: t\geq 0)$ is an isotropic L\'{e}vy process in $\RR ^d$ with the L\'{e}vy--Khintchine
exponent $\psi$. In this Section we prove a Tauberian-like theorem for tails of $\mathbf{X}$. For $t > 0$, we set
\[
	F_t(r)= \PP \big( \abs{X_t} \geq \sqrt{r}\big), \quad r \geq 0.
\]
For a function $f: [0, +\infty) \rightarrow \CC$ its Laplace transform $\calL f$ is defined by
\[
	\calL f(\lambda) = \int_0^\infty e^{-\lambda r} f(r) {\: \rm d}r.
\]
\begin{lemma}
	\label{lem:1}
	If $\psi \in \ROrg$ for some $\alpha \in [0, 2]$ then
	\begin{align*}
		\lim_{\atop{\lambda \to 0^+}{t \psi(\sqrt{\lambda}) \to 0}}
		\frac{\lambda \calL F_t  (\lambda )}{t\psi (\sqrt{\lambda})}
		=2^\alpha \frac{\Gamma \big( (d+\alpha )/2\big)}{\Gamma (d/2)}.
	\end{align*}
\end{lemma}
\begin{proof}
	Let us observe that
	\[
		\lambda \calL F_t (\lambda)
		=
		\int_{\RR^d} \big(1 - e^{-\lambda \norm{x}^2}\big) p(t, {\rm d} x).
	\]
	Since
	\begin{equation}
		\label{eq:25}
		1 - e^{-\norm{x}^2}
		=
		(4 \pi)^{-d/2}
		\int_{\RR^d} \big(1 - \cos \sprod{x}{\xi}\big) e^{-\frac{\norm{\xi}^2}{4}} {\: \rm d}\xi,
	\end{equation}
	by the Fubini--Tonelli's theorem and \eqref{eq:9} we get
	\begin{align*}
		\lambda \calL F_t (\lambda)
		&=
		(4 \pi)^{-d/2}
		\int_{\RR^d} \int_{\RR^d}
		\big(1 - \cos \sprod{x}{\xi \sqrt{\lambda}}\big) p(t, {\rm d} x) 
		e^{-\frac{\norm{\xi}^2}{4}} {\: \rm d}\xi\\
		&=
		(4 \pi)^{-d/2}
		\int_{\RR^d}
		\big(1 - e^{-t \psi(\xi\sqrt{\lambda} )} \big)
		e^{-\frac{\norm{\xi}^2}{4}}
		{\: \rm d}\xi.
	\end{align*}
	Therefore, using polar coordinates we obtain
	\begin{equation}
		\label{eq:4}
		\lambda \calL F_t(\lambda)
		= 
		\frac{2^{1-d}}{\Gamma (d/2)} 
		\int _0^\infty 
		\big(1-e^{-t\psi (r\sqrt{\lambda})}\big)
		e^{-\frac{r^2}{4}} r^{d-1} {\: \rm d}r. 
	\end{equation}
	We claim that for every $\epsilon > 0$ there are  $\delta > 0$ and $C=C(\delta) > 0$  such that for all $r > 0$
	and $0 < u \leq \delta$
	\begin{equation}
		\label{eq:18}
		\psi(r u) \leq C \psi(u) \big(r^2 + r^{-\epsilon}\big).
	\end{equation}
	For $\alpha > 0$, we recall that (see \cite[Theorem 1.5.3]{bgt})
	\[
		\lim_{u \to 0^+} \frac{\psi^*(u)}{\psi(u)} = 1.
	\]
	Hence, by \eqref{eq:2}, there is $C > 0$,
	\[
		\psi(r u) \leq \psi^*(r u) \leq 2 (r^2 + 1) \psi^*(u)
		\leq C (r^2 + 1) \psi(u).
	\]
	For $\alpha = 0$, by \eqref{eq:14}, there is $\delta > 0$ such that if $r u, u \leq \delta$
	then
	\[
		\psi(r u) \leq 2 \psi(u) \max\{r^\epsilon, r^{-\epsilon}\}.
	\]
	Otherwise, $r u \geq \delta$ and \eqref{eq:2} implies that
	\[
		\psi(r u) 
		\leq 
		\psi^*(1) \big(2 + \delta^{-2}\big) r^2 u^2.
	\]
	Again, by \eqref{eq:14}, for every $\epsilon > 0$ there is $C > 0$ such that
	\[
		\psi(u) \geq C u^{\epsilon},
	\]
	we obtain
	
	\[
		\psi(r u) \lesssim r^2 \psi(u) u^{2 - \epsilon}
		\lesssim r^2 \psi(u),
	\]
	proving the claim \eqref{eq:18}.
	
	Next, by \eqref{eq:18}, there is $C > 0$ such that for all $r \geq 0$ and $\lambda \leq \delta$
	\[
		\frac{1 - e^{-t \psi(r \sqrt{\lambda})}}{t \psi(\sqrt{\lambda})}
		\leq
		\frac{\psi(r \sqrt{\lambda})}{\psi(\sqrt{\lambda})}
		\leq
		C (r^2 + r^{-\epsilon}).
	\]
	Hence, by the dominated convergence theorem
	\begin{align*}
		\lim_{\atop{\lambda \to 0^+}{t \psi(\sqrt{\lambda}) \to 0}}
		\frac{1}{t \psi(\sqrt{\lambda})}
		\int_0^\infty \big(1 - e^{-t\psi(r \sqrt{\lambda})}\big) e^{-r^2/4} r^{d-1} {\: \rm d} r
		=
		2^{d+\alpha  -1}\Gamma \big((d+\alpha)/2\big),
	\end{align*}
	because for each $r > 0$
	\[
		\lim_{\atop{\lambda \to 0^+}{t \psi(\sqrt{\lambda}) \to 0}}
		\frac{1 - e^{-t \psi(r \sqrt{\lambda})}}{t \psi(r \sqrt{\lambda})}
		\cdot 
		\frac{\psi(r \sqrt{\lambda})}{\psi(\sqrt{\lambda})} = r^\alpha. \qedhere
	\]
\end{proof}

The following theorem provides the asymptotic behavior of $F_t$ at infinity. Here, we have to exclude
$\alpha = 2$ what is natural since for the Brownian motion $\psi(x)=\norm{x}^2$ and the tail decay exponentially.
\begin{theorem}
	\label{Tail}
	If $\psi \in \ROrg$ for some $\alpha \in [0, 2)$ then 
	\begin{align}
		\label{eq:11}
		\lim_{\atop{r \to 0^+}{t \psi(r) \to 0}} 
		\frac{\PP\big(\norm{X_t} \geq r^{-1}\big)}{t\psi (r)}
		=
		\mathcal{C}_{d, \alpha}
	\end{align}
	where
	\[
		\mathcal{C}_{d, \alpha} = 
		2^{\alpha} 
		\frac{\Gamma\big((d+\alpha)/2\big)}{\Gamma(d/2) \Gamma(1-\alpha/2)}.
	\]
\end{theorem}
\begin{proof}
	Let us define
	\[
		\calF_t(x)=\int _0^x F_t(r) {\: \rm d} r, \quad x \geq 0.
	\]
	Then
	\[
		\calF_t(x)
		\leq e \int_0^x e^{-r/x} F_t(r) {\: \rm d} r
		\leq e  \calL  F_t(x^{-1}\big).
	\]
	Hence, by \eqref{eq:4}
	\begin{equation}
		\label{eq:20}
		\calF_t(x)
		\lesssim
		x t
		\int_0^\infty 
		\psi\big(s x^{-1/2}\big)
		e^{-\frac{s^2}{4}} s^{d-1} {\: \rm d}s.
	\end{equation}
	Let $\delta > 0$. For $(t,r) \in (0, +\infty) \times (0, \delta)$ we define a tempered distribution
	$\Lambda _{t,r}\in \calS^\prime\big([0, +\infty)\big)$ by setting
	\begin{align*}
		\Lambda _{t,r}(f)
		=
		\frac{r}{t \psi\big(r^{1/2}\big)}
		\int _0^\infty 
		f(x) \calF_t\big(r^{-1} x\big) {\: \rm d}x, \quad f \in \mathcal{S}\big([0, \infty)\big).
	\end{align*}
	Recall that the space $\calS\big([0, +\infty)\big)$ consists of Schwartz functions on $\RR$ restricted to
	$[0, +\infty)$ and $\calS^\prime \big([0, +\infty)\big)$ consists of tempered distributions supported by
	$[0, +\infty)$, see \cite{vlad} for details.
	
	We claim that the family $\big(\Lambda _{t,r}: t > 0, r \in (0, \delta)\big)$ is equicontinuous. Indeed,
	by \eqref{eq:20} and \eqref{eq:18}, for each $\epsilon > 0$ there is $\delta > 0$ such that
	\begin{align*}
		\Lambda_{t,r}(f)
		& \lesssim
		\int_0^\infty
		|f(x)|
		x
		\int_0^\infty
		\frac{\psi\big(s r^{1/2} x^{-1/2}\big)}{\psi\big(r^{1/2}\big)} e^{-\frac{s^2}{4}} s^{d-1} {\: \rm d} s
		{\: \rm d} x \\
		& \lesssim
		\int_0^\infty
		|f(x)|
		\int_0^\infty
		\big(s^2 + s^{-\epsilon} x^{1-\epsilon/2}\big) e^{-\frac{s^2}{4}} s^{d-1} {\: \rm d}s
		{\: \rm d}x,
	\end{align*}
	for all $t > 0$ and $r \in (0, \delta)$. Hence,
	\[
		\Lambda_{t, r}(f) \lesssim \sup_{x \geq 0} \big| (1+x)^3 f(x) \big|.
	\]
	Next, for any $\tau > 0$ we set $f_\tau (x)=e^{-\tau x}$. Then
	\[
		\Lambda _{t,r}(f_\tau ) 
		= 
		\frac{1}{t\tau^2 \psi\big(r^{1/2}\big)} \tau r\calL F_t (\tau r)
		=
		\frac{\psi\big(r^{1/2} \tau^{1/2}\big)}{\tau^2 \psi \big(r^{1/2}\big)}
		\cdot
		\frac{\tau r\calL F_t(\tau r)}{t \psi\big(\tau ^{1/2} r^{1/2}\big)}.
	\]
	In particular, by Lemma \ref{lem:1} we obtain
	\begin{align*}
		\lim_{\atop{r \to 0^+}{t \psi(\sqrt{r}) \to 0}}
		\Lambda _{t,r}(f_\tau )
		=C'_{d,\alpha}\tau ^{\alpha /2-2}
		=
		\frac{C'_{d,\alpha}}{\Gamma (2-\alpha /2)}\int _0^\infty e^{-\tau x}x^{1-\alpha /2} {\: \rm d}x
	\end{align*}
	where 
	\[
		C'_{d,\alpha }
		=
		2^\alpha 
		\frac{\Gamma \big((d+\alpha)/2\big)}{\Gamma (d/2)}.
	\]
	Since $\calB$, the linear span of the set $\big\{f_\tau : \tau > 0\big\}$, is dense in $\calS\big([0, +\infty)\big)$
	and the family $\big(\Lambda _{t,r}: t > 0, r \in (0, \delta)\big)$ is equicontinuous on
	$\calS\big([0, +\infty)\big)$, we conclude that for any $f \in \calS\big([0, +\infty)\big)$,
	\begin{align*}
		\lim_{\atop{r \to 0^+}{t \psi(\sqrt{r}) \to 0}}
		\Lambda_{t,r}(f)
		=
		\frac{C'_{d,\alpha}}{\Gamma (2-\alpha /2)}
		\int _0^\infty f(x) x^{1-\alpha /2} {\: \rm d}x.
	\end{align*}
	For a completeness of the argument we provide a sketch of the proof that $\calB$ is dense in
	$\calS\big([0, +\infty)\big)$, for details we refer to \cite{vlad} and \cite{vdz}. By the Hahn--Banach theorem
	it is enough to show that if $\Lambda \in \calS^\prime\big([0, +\infty)\big)$ and $\Lambda (\phi)=0$ for all
	$\phi \in \calB$ then $\Lambda$ is the zero functional. Assume that $\Lambda $ vanishes on $\calB$ and let
	$\calL \Lambda (z)=\Lambda (e^{-z\cdot})$, $\mathrm{Re}\, z>0$, be the Laplace transform of $\Lambda$. Since
	$\Lambda =0$ on $\calB$ we get $\calL \Lambda (\lambda )=0$ for $\lambda >0$. But the Laplace transform is analytic
	in the half-plane $\mathrm{Re}\, z>0$, whence $\calL \Lambda (z)=0$, for $\mathrm{Re}\,z>0$. Using the connection
	between Laplace and Fourier transforms,
	\begin{align*}
		\widehat{\Lambda}(\xi)
		=
		\lim_{\lambda \to 0^+}\calL \Lambda (\lambda +i\xi),\quad \mathrm{in}\ \calS^\prime\big([0, +\infty)\big),
	\end{align*}
	we obtain that $\widehat{\Lambda}=0$. Hence, $\Lambda $ is the zero functional as desired. 
	
	Now, we claim that 
	\begin{equation}
		\label{eq:10}
		\lim_{\atop{r \to 0^+}{t \psi(\sqrt{r}) \to 0}}
		\frac{r \calF_t\big(r^{-1}\big)}{t \psi(\sqrt{r})} 
		= \frac{C'_{d,\alpha}}{\Gamma(2-\alpha/2)}.
	\end{equation}
	For a given $\epsilon > 0$ choose $\phi_+ \in \calS\big([0, +\infty)\big)$ such that 
	$$
	\phi_+(x) =
	\begin{cases}
		1 & \text{ for } 0 \leq x \leq 1, \\
		0 & \text{ for } 1+\varepsilon \leq x.
	\end{cases}
	$$
	We have
	\begin{align*}
		\frac{r \calF_t\big(r^{-1}\big)}{t \psi(\sqrt{r})}
		= 
		\frac{r}{ t \psi(\sqrt{r})} \int_0^{1/r} {\: \rm d} \calF_t(s) 
		&\leq 
		\frac{r}{ t \psi(\sqrt{r})} \int_0^{1/r} \phi_+(s/r) {\: \rm d}\calF_t(s)\\
		&\leq
		\frac{r}{ t \psi(\sqrt{r})} \int_0^{\infty} \phi_+(s/r) {\: \rm d}\calF_t(s),
	\end{align*}
	thus
	$$
	\frac{r \calF_t\big(r^{-1}\big)}{t \psi(\sqrt{r})}
	\leq 
	- \Lambda_{t, r}(\phi_+').
	$$
	Hence,
	\begin{align*}
		\limsup_{\atop{r \to 0^+}{t \psi(\sqrt{r}) \to 0}} 
		\frac{r \calF_t\big(r^{-1}\big)}{t \psi(\sqrt{r})}
		&\leq 
		\frac{-C'_{d,\alpha}}{\Gamma(2-\alpha/2)} 
		\int_0^{\infty} s^{1 - \alpha/2} \phi'_+(s) {\: \rm d}s\\
		&=\frac{C'_{d,\alpha}}{\Gamma(1-\alpha/2)}
		\int_0^\infty 
		s^{ - \alpha/2} \phi_+(s) {\: \rm d}s \\
		& \leq
		\frac{C'_{d,\alpha}}{\Gamma(2 - \alpha/2)} (1+\varepsilon)^{1-\alpha/2}.
	\end{align*}
	Similarly, taking $\phi_- \in \calS\big([0, +\infty)\big)$ such that
	$$
	\phi_-(x) =
	\begin{cases}
		1 & \text{ for } 0 \leq x \leq 1-\varepsilon, \\
		0 & \text{ for } 1 \leq x,
	\end{cases}
	$$
	we may show that
	$$
	\liminf_{\atop{r \to 0^+}{t \psi(\sqrt{r}) \to 0}}
	\frac{r \calF_t\big(r^{-1}\big)}{t\psi(\sqrt{r})}
	\geq 
	\frac{C'_{d,\alpha}}{\Gamma(2 - \alpha/2)} (1 - \varepsilon)^{1 - \alpha/2}.
	$$
	This proves \eqref{eq:10}.

	To show \eqref{eq:11} we adapt the proof of the monotone density theorem (see e.g. \cite[Theorem 1.7.2]{bgt}).
	Let $\epsilon > 0$. The function $F_t(s)$ is non-increasing therefore
	\begin{equation}
		\label{eq:12}
		\calF_t\big(r^{-1}\big) - \calF_t\big((1-\epsilon) r^{-1}\big)
		=
		\int_{(1-\epsilon)/r}^{1/r} F_t(s) {\: \rm d} s
		\geq \epsilon r^{-1} F_t\big(r^{-1}\big),
	\end{equation}
	and
	\begin{align}
		\label{eq:12a}
		\calF_t\big((1+\epsilon)r^{-1}\big) - \calF_t\big(r^{-1}\big)
		= \int^{(1+\epsilon)/r}_{1/r} F_t(s) {\: \rm d}s 
		\leq \epsilon r^{-1} F_t\big(r^{-1}\big).
	\end{align}
	By \eqref{eq:10},
	$$
	\lim_{\atop{r \to 0^+}{t \psi(\sqrt{r}) \to 0}}
	\frac{r \calF_t \big((1-\epsilon) r^{-1}\big)}{t \psi(\sqrt{r})} 
	=\frac{C'_{d,\alpha}(1-\epsilon)^{1-\alpha/2}}{\Gamma(2 - \alpha/2)}.
	$$
	Hence, \eqref{eq:12} implies
	$$
	\limsup_{\atop{r \to 0^+}{t \psi(\sqrt{r}) \to 0}}
	\frac{F_t\big(r^{-1}\big)}{t\psi(\sqrt{r})}
	\leq
	\frac{C'_{d,\alpha}}{\Gamma(2 - \alpha/2)}
	\cdot
	\frac{1 - (1-\epsilon)^{1-\alpha/2}}{\epsilon }.
	$$
	Similarly, by \eqref{eq:12a} we get
	$$
	\liminf_{\atop{r \to 0^+}{t \psi(\sqrt{r}) \to 0}}
	\frac{F_t\big(r^{-1}\big)}{t \psi(\sqrt{r})}
	\geq 
	\frac{C'_{d,\alpha}}{\Gamma(2 - \alpha/2)}
	\cdot
	\frac{(1+\epsilon)^{1-\alpha/2} -1}{\epsilon }.
	$$
	Finally, by taking $\epsilon$ tending to zero we obtain \eqref{eq:11}.
\end{proof}

By the same line of reasoning as in proofs of Lemma \ref{lem:1} and Theorem \ref{Tail}
we may show the corresponding results if the L\'{e}vy--Khintchine exponent $\psi$ varies regularly at infinity.
\begin{lemma}
	\label{lem:2}
	If $\psi \in \RInf$ for $\alpha \in [0, 2]$ then
	\[
		\lim_{\atop{\lambda \to +\infty}{t \psi(\sqrt{\lambda}) \to 0}}
		\frac{ \lambda \calL F_t (\lambda )}{t\psi (\sqrt{\lambda})} =  
		2^{\alpha}
		\frac{\Gamma\big((d+\alpha)/2\big)}{\Gamma(d/2)}.
	\]
\end{lemma}

\begin{theorem}
	\label{Tail2}
	If $\psi \in \RInf$ for some $\alpha \in [0, 2)$ then
	\[
		\lim_{\atop{r \to +\infty}{t \psi(r) \to 0}}
		\frac{\PP\big(\norm{X_t} \geq r^{-1}\big)}{t\psi (r)}
		=
		\mathcal{C}_{d, \alpha}.
	\]
\end{theorem}

\section{Asymptotic behavior of the densities}
Suppose that $\mathbf X=(X_t : t\geq 0)$ is an isotropic \emph{unimodal} L\'{e}vy process in $\RR ^d$,
i.e. a process having a rotationally invariant and radially non-increasing density function $p(t, \:\cdot\:)$
on $\RR ^d\setminus \{0\}$. 

\subsection{Transition density asymptotic}
In the following theorem we give the asymptotic behavior of the transition density $p(t, x)$.
\begin{theorem}
	\label{densityAsymp}
	If $\psi \in \ROrg$, for some $\alpha \in (0, 2)$, then 
	\begin{align*}
		\lim_{\atop{\norm{x} \to +\infty}{t \psi(\norm{x}^{-1}) \to 0}}
		\frac{p(t, x)}{\norm{x}^{-d} t\psi \big(\norm{x}^{-1}\big)}
		=
		\mathcal{A}_{d, \alpha}
	\end{align*}
		where
	\begin{align*}
		\mathcal{A}_{d, \alpha}
		=
		\alpha 2^{\alpha -1}\pi ^{-d/2 -1}
		\sin\Big(\frac{\alpha \pi}{2}\Big)
		\Gamma\Big(\frac{\alpha}{2}\Big)
		\Gamma\bigg(\frac{d+\alpha}{2}\bigg).
	\end{align*}
\end{theorem}
\begin{proof}
	In the proof we adapt the argument from \cite[Theorem 1.7.2]{bgt}. For any $0<a<b$, we have
	\[
		F_t(a/r)-F_t(b/r)
		=
		c_d\int_{\sqrt{a/r}}^{\sqrt{b/r}} u^{d-1} p(t, u) {\: \rm d}u,
	\]
	where 
	\[
		c_d=2 \frac{\pi^{d/2}} {\Gamma (d/2)}.
	\]
	Since the function $u \mapsto p(t, u)$ is nonincreasing, we get
	\begin{equation}
		\label{eq:19}
				\frac{F_t(a/r)-F_t(b/r)} {t\psi (\sqrt{r})}  \geq
		\frac{c_d}{d}
		\cdot \frac{p(t, \sqrt{b/r})} {t\psi (\sqrt{r})} 
		\cdot \frac{b^{d/2}-a^{d/2}}{r^{d/2}}, \end{equation}
and \begin{equation}\label{eq:19a}
		\frac{F_t(a/r)-F_t(b/r)} {t\psi (\sqrt{r})}  
		 \leq 
		\frac{c_d}{d}
		\cdot \frac{p(t, \sqrt{a/r})} {t\psi (\sqrt{r})} 
		\cdot \frac{b^{d/2}-a^{d/2}}{r^{d/2}}
		.
		\end{equation}
	By Theorem \ref{Tail}, we have
	\begin{align*}
		\lim_{\atop{r \to 0^+}{t \psi(\sqrt{r}) \to 0}}
		\frac{F_t(a/r)-F_t(b/r)}{t \psi (\sqrt{r})}
		=
		\mathcal{C}_{d,\alpha } \big(a^{-\alpha/2}-b^{-\alpha/2}\big).
	\end{align*}
	Hence,  \eqref{eq:19} gives
	\begin{align*}
		\limsup_{\atop{r \to 0^+}{t \psi(\sqrt{r}) \to 0}}
		\frac{p(t, \sqrt{b/r})} {r^{d/2} t \psi (\sqrt{r})}
		\leq
		d  
		\cdot \frac{\mathcal{C}_{d, \alpha}}{c_d}
		\cdot \frac{a^{-\alpha /2}-b^{-\alpha /2}}{b^{d/2}-a^{d/2}}.
	\end{align*}
	Taking $b=1$, $a=1-\epsilon $ and letting $\epsilon$ to zero we obtain
	\begin{align*}
		\lim_{\atop{r \to 0^+}{t \psi(\sqrt{r}) \to 0}}
		\frac{p\big(t, r^{-1/2}\big)} {r^{d/2} t \psi (\sqrt{r})}
		\leq 
		\alpha \cdot \frac{\mathcal{C}_{d,\alpha}}{c_d}.
	\end{align*}
	Similarly, using \eqref{eq:19a}, one can show that
	\begin{align*}
		\liminf_{\atop{r \to 0^+}{t \psi(\sqrt{r}) \to 0}}
		\frac{p\big(t, r^{-1/2}\big)} {r^{d/2} t\psi (\sqrt{r})}
		\geq
		\alpha \cdot \frac{\mathcal{C}_{d, \alpha}}{c_d}.
	\end{align*}
	Finally, by the Euler's reflection formula we conclude the proof.
\end{proof}

Using Lemma \ref{lem:2} and Theorem \ref{Tail2} we obtain the following asymptotic.
\begin{theorem}
	\label{densityAsymp1}
	If $\psi \in \RInf$ for some $\alpha \in (0, 2)$ then 
	\begin{align*}
		\lim_{\atop{x \to 0}{t \psi(\norm{x}^{-1}) \to 0}}
		\frac{p(t, x)}{\norm{x}^{-d} t \psi\big(\norm{x}^{-1}\big)}
		=
		\mathcal{A}_{d, \alpha}.
	\end{align*}
\end{theorem}

Next, let us denote by $\nu(x)$ the density function of the L\'{e}vy measure $\nu$ associated to the process $\mathbf X$.
The following theorem gives the asymptotic of $\nu(x)$ as well as the equivalence of asymptotics of $p(t, x)$ and
$\nu(x)$ with regular variation of the L\'{e}vy--Khintchine exponent.
\begin{theorem}
	\label{LevyDensity}
	Let $\mathbf X = (X_t: t \geq 0)$ be an isotropic unimodal L\'{e}vy process on $\RR^d$ with the characteristic
	exponent $\psi$ and the L\'{e}vy density $\nu$. Then the following are equivalent:
	\begin{enumerate}
		\item $\psi \in \RInf$ for some $\alpha\in(0,2)$;
		\item there is $c > 0$,
		\[
			\lim_{\atop{x \to 0}{t \psi(\norm{x}^{-1}) \to 0}}
			\frac{p(t, x)}{\norm{x}^{-d} t \psi\big(\norm{x}^{-1}\big)}
			=c;
		\]
		\item there is $c > 0$,
		\[
			\lim_{x \to 0}
			\frac{\nu(x)}{\norm{x}^{-d} \psi\big(\norm{x}^{-1}\big)}=c.
		\]
	\end{enumerate}
\end{theorem}
\begin{proof}
	We observe that Theorem \ref{densityAsymp1} yields the implication (i) $\Rightarrow$ (ii).
	Also the implication (ii) $\Rightarrow$ (iii) follows because 
	\begin{equation}
		\label{eq:23}
		\lim_{t \to 0^+}
		t^{-1} p(t,x)
		=
		\nu(x)
	\end{equation}
	vaguely on $\RR^d\setminus\{0\}$. Indeed, let $\epsilon > 0$, then there exists $\delta > 0$ such that
	for $\norm{x}\leq \delta$ and $t \psi(\norm{x}^{-1}) < \delta$
	\[
		c - \epsilon \leq \frac{t^{-1} p(t, x)}{\norm{x}^{-d} \psi\big(\norm{x}^{-1}\big)} \leq c + \epsilon.
	\]
	Hence, by taking $t$ approaching zero we get
	\[
		c - \epsilon \le 
		\frac{\nu(x)}{\norm{x}^{-d} \psi\big(\norm{x}^{-1}\big)}
		\leq 
		c + \epsilon
	\]
	for $0 < \norm{x} \leq \delta$.
	
	To prove that (iii) implies (i), we use Drasin--Shea Theorem (see \cite[Theorem 6.2]{shea}, see also 
	\cite[Theorem 5.2.1]{bgt}).	Let $u_0=\big(d^{-1/2},\ldots, d^{-1/2}\big) \in \mathbb{R}^d$ and $r>0$. Notice that
	(iii) forces the Gaussian part to vanish. Using the polar coordinates we may write
	\begin{align*}
		\psi(r) = \psi(r u_0)
		&=
		\int^\infty_0 \int_{\Ss^{d-1}} 
		\big(1-\cos( \rho r \sprod{u_0}{u})\big) \sigma({\rm d} u)
		\rho^{d-1} \nu(\rho) {\: \rm d}\rho\\
		&=
		\int^\infty_0 k(\rho^{-1} r) \rho^{-d} \nu\big(\rho^{-1}\big) \rho^{-1} {\: \rm d} \rho
	\end{align*}
	where $\sigma$ denotes the spherical measure on the unite sphere $\Ss^{d-1}$ in $\RR^d$, and
	\begin{equation}
		\label{eq:40}
		k(r) = \int_{\Ss^{d-1}} \big(1 - \cos(r \sprod{u_0}{u})\big) \sigma({\rm d} u).
	\end{equation}
	Let us recall the definition of the \emph{Mellin convolution}, see e.g. \cite[Section 4.1]{bgt}, defined
	for two functions $f, g : [0, \infty) \rightarrow \mathbb{C}$ by the formula 
	\begin{align*}
		\calM(f,g)(x) = \int _0^\infty f\big(t^{-1} x\big) g(t) t^{-1} {\: \rm d} t.
	\end{align*}
	Then, by setting $f(r)=r^{-d} \nu\big(r^{-1}\big)$, we may write 
	\begin{align*}
		\psi(r) = \calM(k,f)(r).
	\end{align*}
	Since 
	\begin{equation}\label{eq:50}
		0 \leq k(r) \leq (1 \wedge r^2) \sigma\big(\Ss^{d-1}\big)
	\end{equation}
	the Mellin transform $\check{k}$ where
	\[ 
		\check{k}(z)=\int^\infty_0 t^{-z-1} k(t) {\: \rm d} t,
	\]
	is absolutely convergent on the strip $\{z \in \CC : 0< \mathrm{Re}\, z < 2\}$. Moreover, $x \mapsto \nu(x)$
	is nonincreasing and integrable on $\big\{x \in \RR^d: \norm{x} \geq 1\big\}$, thus
	\[
		\lim_{r \to 0^+} f(r)=0.
	\]
	By (iii), there is $\delta > 0$ such that for all $\norm{x} \leq \delta$
	\[
		(c - \epsilon) \norm{x}^{-d} \psi\big(\norm{x}^{-1}\big) 
		\leq
		\nu(x) 
		\leq
		(c + \epsilon) \norm{x}^{-d} \psi\big(\norm{x}^{-1}\big).
	\]
	Hence, \cite[Theorem 26]{MR3165234} implies that $\psi$ satisfies weak upper and lower scaling, i.e.
	there are $C > 0$, and $\overline{\beta}, \underline{\beta} \in (0, 2)$, and $r_0 \geq 0$ such that for all
	$r \geq r_0$,
	\[
		C^{-1} r^{\underline{\beta}} \leq \psi(r) \leq C r^{\overline{\beta}}.
	\]
	Therefore, if $r \geq \max\big\{\delta^{-1}, r_0 \big\}$ we obtain
	\[
		C^{-1} (c-\epsilon) r^{\underline{\beta}} 
		\leq 
		r^{-d} \nu\big(r^{-1}\big)
		\leq C (c+\epsilon) r^{\overline{\beta}}.
	\]
	Thus,
	\[
		\rho = \limsup_{r \to +\infty} \frac{\log f(r)}{\log r} \in (0, 2),
	\]
	and $f$ has bounded decrease (see \cite[Section 2.1]{bgt}). Moreover,
	\[
		\lim_{r \to +\infty}
		\frac{\calM(k, f)(r)}{f(r)} 
		=
		\lim_{x \to 0}
		\frac{\psi\big(\norm{x}^{-1}\big)}{\norm{x}^{-d} \nu(x)} 
		=
		c^{-1}.
	\]
	Therefore, we may apply the Drasin--Shea theorem to conclude that $f \in \calR_\rho^\infty$ what translates
	to $\psi \in \calR_\rho^\infty$.
\end{proof}

\begin{corollary}
	If $\psi \in \RInf$ then
	\[
		\lim_{\atop{x \to 0}{t \psi(\norm{x}^{-1}) \to 0}} \frac{p(t, x)}{t \nu(x)} = 1.
	\]
\end{corollary}
To prove an analogue of the above theorem in the case when $\psi \in \ROrg$ we shall need the following lemma.
\begin{lemma}
	\label{Lemma_scaling}
	Suppose that there are $c > 0$ and $M > 1$ such that
	\begin{align*}
		\nu(x) \geq c \norm{x}^{-d} \psi\big(\norm{x}^{-1}\big)
	\end{align*}
	whenever $\norm{x} \geq M$. Then there are $C > 1$, and $\underline{\beta}, \overline{\beta} \in (0, 2)$
	such that for all $0 < \lambda, r \leq 1$
	\begin{equation}
		\label{weak_scaling0}
		C^{-1} \lambda^{\underline{\beta}} \psi(r)
		\leq \psi(\lambda r)
		\leq C \lambda^{\overline{\beta}} \psi(r).
	\end{equation}
\end{lemma}
\begin{proof}
	We adapt the proof of \cite[Theorem 26]{MR3165234} to the current settings. First, let us notice that
	\begin{equation}
		\label{lim0psi}
		\lim_{r \to 0^+} 
		\frac{r^2}{\psi(r)}=0.
	\end{equation}
	Indeed, by \cite[Corollary 1]{MR3225805}, for $0 < r<M^{-1}$, 
	\[
		\psi^*(r)
		\geq 
		\frac{r^2}{24d} \int_{r \norm{x} \leq 1} \norm{x}^2 \nu(x) {\: \rm d}x
		\gtrsim 
		r^2 \int^{r^{-1}}_{M} u \psi\big(u^{-1}\big) {\: \rm d}u.
	\]
	Moreover, by \eqref{eq:2} and \eqref{eq:27}, for $u \geq 1$,
	\[
		\pi^2 \psi(u^{-1})\geq \psi^*(u^{-1})\geq \frac{1}{2(1+u^2)}\psi^*(1).
	\]
	Thus,
	\[
		\psi(r)
		\gtrsim 
		r^2
		\int^{r^{-1}}_{M}\frac{u}{1+u^2} {\: \rm d}u,
	\]
	what implies \eqref{lim0psi}.
	
	In view of \eqref{lim0psi}, we may assume that $\mathbf X$ is pure-jump, i.e. n \eqref{eq:28}, $\eta = 0$. Let
	us consider a function $\phi: [0, +\infty) \rightarrow \RR$ defined by
	\begin{align*}
		\phi (\lambda) = \int_0^{\infty} \frac{\lambda}{\lambda + s} \nu\big(s^{-1/2}\big) s^{-1-d/2} {\: \rm d}s.
	\end{align*}
	Then
	\[
		\phi(\lambda) = \int_0^\infty \big(1 - e^{-\lambda r}\big) \mu(r) {\: \rm d}r,
	\]
	where
	\[
		\mu(r) = \int_0^\infty e^{-r s} \nu\big(s^{-1/2}\big) s^{-d/2} {\: \rm d} s.
	\]
	By \cite[the proof of Theorem 6.2]{ssv}, the function $\phi$ is a complete Bernstein function. According to
	\cite[estimates (28)--(30)]{MR3165234}, there is $C > 0$ such
	\begin{equation}
		\label{eq:22}
		C^{-1} \phi(\lambda) \leq \psi(\sqrt{\lambda}) \leq C \phi(\lambda), \quad \text{for all } \lambda > 0,
	\end{equation}
	and
	\[
		\mu(r) \leq C r^{-2} \phi'\big(r^{-1}\big), \quad \text{for all } r > 0,
	\]
	and
	\[
		\nu(x) \leq C \norm{x}^{-d+2} \mu\big(\norm{x}^2\big), \quad \text{for all } x \neq 0.
	\]
	Therefore, for $\norm{x} > M$ we get for some $c>0$
	\begin{equation}
		\label{eq:13}
		c C^{-1} \phi\big(\norm{x}^{-2}\big)
		\leq 
		c \psi\big(\norm{x}^{-1}\big)
		\leq
		\norm{x}^d \nu (x)
		\leq
				C^2 \norm{x}^{-2} \phi'\big(\norm{x}^{-2}\big).
	\end{equation}
	Hence, there is $\overline{\beta} > 0$ such that for all $\lambda \in \big(0, M^{-2}\big)$,
	\[
		\overline{\beta} \phi(\lambda) \leq \lambda \phi'(\lambda).
	\]
	In particular, the function $\lambda \mapsto \lambda^{-\overline{\beta}} \phi(\lambda)$ is nondecreasing on
	$\big(0,M^{-2}\big)$, thus, for all $u \in (0, 1)$ and $\lambda \in \big(0, M^{-2}\big)$
	\[
		(u \lambda)^{-\overline{\beta}} \phi(u \lambda) \leq \lambda^{-\overline{\beta}} \phi(\lambda), 
	\]
	what implies the upper bound of \eqref{weak_scaling0} due to \eqref{eq:22} and continuity of $\psi$. From
	\eqref{lim0psi} we may deduce that $\overline{\beta} < 2$.

	Next, we show the lower scaling at zero. As $\phi$ is a complete Bernstein function, we have that
	$\phi_1 : [0, +\infty) \rightarrow \RR$ where
	\[
		\phi _1(\lambda )=\frac{\lambda}{\phi (\lambda)},
	\]
	is a special Bernstein function. Since $\mathbf X$ is pure-jump L\'{e}vy process, 
	\[
		\lim_{\norm{x} \to +\infty} \frac{\psi(x)}{\norm{x}^2} = 0,
	\]
	Thus, by \eqref{eq:22} we conclude
	\[
		\lim_{\lambda \to +\infty} \phi_1(\lambda) = +\infty.
	\]
	Moreover, $\phi(0) = 0$. Hence, the potential measure of the subordinator with
	the Laplace exponent $\phi_1$ (see \cite[(10.9) and Theorem 10.3]{ssv})	is absolutely continuous with
	the density function 
	\begin{align*}
		f(s)=\int _s^\infty \mu(u) {\: \rm d}u.
	\end{align*}
	We observe that, by \eqref{eq:13}, for $s > M^2$,
	\[
		\mu(s) \geq c C^{-2} \phi\big(s^{-1}\big),
	\]
	thus,
	\begin{align}
		\nonumber
		f(s)
		& \geq
		c C^{-2}
		\int_s^\infty \phi \big(u^{-1}\big) u^{-1} {\: \rm d}u \\
		\nonumber
		& \geq 
		cC^{-2} \int _s^\infty \phi'\big(u^{-1}\big) u^{-2} {\: \rm d}u \\
		\label{eq:24}
		& =
		c C^{-2} \phi\big(s^{-1}\big). 
	\end{align}
	Since $\calL f = 1/\phi_1$, by \cite[Lemma 5]{MR3165234}, there is $D > 0$ such that for $s > 0$,
	\[
		f(s) \leq D \frac{\phi_1'\big(s^{-1}\big)}{s^2 \phi_1^2\big(s^{-1}\big)}.
	\]
	Hence, by \eqref{eq:24}, there is $\underline{\beta} > 0$ such that for $\lambda \in \big(0, M^{-2}\big)$
	\begin{equation}
		\label{eq:26}
		\underline{\beta} \phi _1(\lambda) \leq \lambda \phi_1'(\lambda).
	\end{equation}
	Therefore, for all $u \in (0, 1)$ and $\lambda \in \big(0, M^{-2}\big)$
	\[
	 	\frac{(u\lambda)^{1-\underline{\beta}}}{\phi(u\lambda)} =
		(u \lambda)^{-\underline{\beta}} \phi_1(u \lambda) \leq \lambda^{-\underline{\beta}} \phi_1(\lambda)
		=
		\frac{\lambda^{1-\underline{\beta}}}{\phi(\lambda)},
	\]
	what implies the left inequality in \eqref{weak_scaling0}. Let us observe that since $\phi_1 $ is concave,
	we have
	\[
		\lambda \phi_1 '(\lambda)\leq \phi_1 (\lambda),
	\]
	thus, \eqref{eq:26} forces $\underline{\beta} < 1$.
\end{proof}

\begin{theorem}
	\label{LevyDensity2}
	Let $\mathbf X = (X_t : t \geq 0)$ be an isotropic unimodal L\'{e}vy process on $\RR^d$ with
	the characteristic exponent $\psi$ and the L\'{e}vy density $\nu$. Then the following are equivalent:
	\begin{enumerate}
		\item $\psi\in \ROrg$, for some $\alpha\in(0,2)$;
		\item there is $c>0$
		\[
			\lim _{\atop{\norm{x} \to +\infty}{t \psi(|x|^{-1}) \to 0}}
			\frac{p(t, x)}{\norm{x}^{-d} t\psi\big(\norm{x}^{-1}\big)}=c;
		\]
	\item there is $c>0$
		\[
			\lim_{\norm{x} \to +\infty}
			\frac{\nu(x)} {\norm{x}^{-d} \psi\big(\norm{x}^{-1}\big)}=c.
		\]
	\end{enumerate}
\end{theorem}
\begin{proof}
	The proof is similar to Theorem \ref{LevyDensity}. First, in view of Theorem \ref{densityAsymp} and \eqref{eq:23},
	it is enough to show that (iii) implies (i).

	Here, again we use the Drasin--Shea theorem. In light of Lemma \ref{Lemma_scaling} we may assume that the process
	$\mathbf X$ is pure-jump. Let $u_0 = \big(d^{-1/2}, \ldots, d^{-1/2}\big) \in \RR^d$. For $r > 0$ we can write
	\[
		\psi\big(r^{-1}\big) = \calM(K, F)(r)
	\]
	where $F(r) = r^d \nu(r)$ and
	\[
		K(r) = \int_{\Ss^{d-1}} \big(1 - \cos\big(r^{-1} \sprod{u_0}{u}\big)\big) \sigma({\rm d} u).
	\]
	Notice that the function $F$ does not vanish at zero. Therefore, instead of $F$ we consider
	\[
		\tilde{F}(r) = 
		\begin{cases}
			F(r) & \text{if } r \geq 1,\\
			0    & \text{otherwise.}
		\end{cases}
	\]
	Since	
	\[
		0 \leq \calM(K, F)(r) - \calM(K, \tilde{F})(r) \leq C r^{-2} \int_0^1 t^{1+d} \nu(t) {\: \rm d}t,
	\]
	by \eqref{lim0psi}, we have
	\[
		\lim_{r \to +\infty}
		\frac{\calM(K, F)(r)}{\calM(K, \tilde{F})(r)} = 1.
	\]
	By (iii), for $\epsilon > 0$ there is $M_0 > 1$ such that for all $\norm{x} \geq M_0$,
	\[
		(c - \epsilon) \norm{x}^{-d} \psi\big(\norm{x}^{-1}\big)
		\leq
		\nu(x)
		\leq
		(c + \epsilon) \norm{x}^{-d} \psi\big(\norm{x}^{-1}\big).
	\]
	Hence, by Lemma \ref{Lemma_scaling}, there are $C > 1$, and $\underline{\beta}, \overline{\beta} \in (0, 2)$,
	such that for all $\norm{x} \geq M_0$,
	\[
		C^{-1} (c-\epsilon) \norm{x}^{-\underline{\beta}} \psi\big(M_0^{-1}\big)
		\leq
		\norm{x}^d \nu(x) 
		\leq 
		C (c+\epsilon) \norm{x}^{-\overline{\beta}} \psi\big(M_0^{-1}\big).
	\]
	In particular,
	\[
		\rho = \limsup_{r \to +\infty} \frac{\log \tilde{F}(r)}{\log r} \in (-2, 0),
	\]
	and $\tilde{F}$ has bounded decrease. Again, by (iii), we get
	\begin{align*}
		\lim_{r \to +\infty}
		\frac{\calM(K,\tilde{F})(r)}{\tilde{F}(r)} 
		=
		\lim_{r \to +\infty} \frac{\calM(K, F)(r)}{F(r)}
		=
		\lim _{\norm{x} \to +\infty } \frac{\norm{x}^{-d} \psi\big(\norm{x}^{-1} \big)}{\nu (x)}=c^{-1}.
	\end{align*}
	Now, we may apply the Drasin--Shea theorem to obtain $\tilde{F} \in \calR_\rho^\infty$. Therefore,
	$F \in \calR_\rho^\infty$ what implies $\psi \in \calR_\rho^0$.
\end{proof}

\begin{corollary}
	If $\psi \in \ROrg$ then
	\[
		\lim_{\atop{\norm{x} \to +\infty}{t \psi(\norm{x}^{-1}) \to 0}} \frac{p(t, x)}{t \nu(x)} = 1.
	\]
\end{corollary}

\subsection{Green function asymptotic}
\label{sec:4.2}
In this subsection we assume $d \geq 3$. Therefore, the isotropic unimodal L\'{e}vy process $\mathbf X = (X_t: t \geq 0)$
is transient and the associated potential measure $G$ is well-defined,
\begin{align*}
	G(x, A)=\int _{0}^{\infty } \PP_x(X_t \in A) {\: \rm d}t,
\end{align*}  
where $\PP_x$ is the standard measure $\PP(\; \cdot\; | X_0 = x)$ and $A \subset \RR^d$ is a Borel set. We set
$G(A)=G(0,A)$. We also use the same notation $G$ for the density of the part of the potential measure absolutely
continuous with respect to the Lebesgue measure. Thus, we have $G(x,y)= G(0,y-x)$. We set $G(x)=G(0,x)$.
\begin{theorem}
	\label{GreenAsymp}
	Assume that $\psi \in \ROrg$, for some $\alpha \in [0, 2]$. Then
	\begin{align*}
		\lim _{r \to +\infty}
		\psi\big(r^{-1}\big) G\big(\{x : \norm{x} \leq r\}\big)
		= \mathcal{\tilde{C}}_{d, \alpha},
	\end{align*}
	where
	\[
		\mathcal{\tilde{C}}_{d,\alpha} =
		2^{-\alpha} \frac{\Gamma\big((d-\alpha)/2\big)}{\Gamma(d/2) \Gamma(1+\alpha /2 )}.
	\]
\end{theorem}
\begin{proof}
	Let us define
	\begin{align*}
		f(r)
		=
		G\big(\{x : \norm{x} \leq \sqrt{r}\}\big)
		=
		\int_0^\infty 
		\int_{\norm{x} \leq \sqrt{r}} p(t, {\rm d} x)
		{\: \rm d}t.
	\end{align*}
	Hence, by the Fubini--Tonelli's theorem
	\begin{align*}
		\calL f(\lambda)
		=
		\lambda^{-1}
		\int_0^\infty
		\int_{\RR^d}
		e^{-\lambda \norm{x}^2}
		p(t, {\rm d}x)
		{\: \rm d}t.
	\end{align*}
	By \eqref{eq:25}, the second application of the Fubini--Tonelli's theorem gives
	\begin{align*}
		\calL f (\lambda) 
		&=
		(4 \pi)^{-d/2}
		\lambda^{-1}
		\int_0^\infty
		\int_{\RR^d}
		\int_{\RR^d}
		\cos \sprod{x}{\xi\sqrt{\lambda}}
		\: p(t, {\rm d}x)
		e^{-\frac{\norm{\xi}^2}{4}} 
		{\: \rm d}\xi
		{\: \rm d}t\\
		&=
		(4\pi)^{-d/2}
		\lambda^{-1}
		\int_0^\infty
		\int_{\RR^d}
		e^{-t \psi(\xi \sqrt{\lambda})}
		e^{-\frac{\norm{\xi}^2}{4}} 
		{\: \rm d}\xi
		{\: \rm d}t
	\end{align*}
	where in the last equality we have used \eqref{eq:9}. Finally, after integration with respect to $t$ 
	with a help of polar coordinates we may write
	\begin{align*}
		\calL f(\lambda) 
		&=
		(4 \pi)^{-d/2} \lambda^{-1}
		\int_{\RR^d}
		e^{-\frac{\norm{\xi}^2}{4}} \frac{{\rm d}\xi}{\psi(\xi \sqrt{\lambda})}\\
		&=
		\frac{2^{1-d}}{\Gamma(d/2)}
		\lambda^{-1}
		\int_0^\infty
		e^{-\frac{r^2}{4}} r^{d-1} \frac{{\rm d}r}{\psi(r \sqrt{\lambda})}.
	\end{align*}
	Since the process $\mathbf X$ is unimodal, $\psi$ and $\psi^*$ are comparable, by \eqref{eq:27}. Therefore,
	by \eqref{eq:2}
	\[
		\psi(\sqrt{\lambda}) \leq \psi^*(\sqrt{\lambda}) \leq 2(r^{-2} + 1) \psi^*(r \sqrt{\lambda})
		\leq 2\pi^2 (r^{-2} + 1) \psi(r \sqrt{\lambda}).
	\]
	Because $d \geq 3$, by the dominated convergence we obtain
	\[
		\lim_{\lambda \to 0^+}
		\frac{\lambda \calL f(\lambda)}{\psi(\sqrt{\lambda})} 
		= 2^{-\alpha} \frac{\Gamma\big((d-\alpha)/2\big)}{\Gamma(d/2)}.
	\]
	Hence, by a Tauberian theorem (see e.g. \cite[Theorem 1.7.1]{bgt}) we conclude the proof.
\end{proof}

\begin{corollary}
	\label{cor:1}
	Assume that $\psi \in \ROrg$, for some $\alpha \in (0, 2]$. Then
	\begin{align*}
		\lim _{\norm{x} \to +\infty} \norm{x}^{d} \psi\big(\norm{x}^{-1} \big) G(x) = \mathcal{\tilde{A}}_{d, \alpha},
	\end{align*}
	where
	\[
		\mathcal{\tilde{A}}_{d, \alpha} = 2^{-\alpha} \pi^{-d/2} \frac{\Gamma\big((d-\alpha)/2\big)}{\Gamma(\alpha/2)}.
	\]
\end{corollary}
\begin{proof}
	For the proof we use Theorem \ref{GreenAsymp} and the line of reasoning from the proof of Theorem \ref{densityAsymp}.
\end{proof}

\begin{theorem}
	\label{GreenAsymp2}
	Assume that $\psi \in \RInf$, for some $\alpha \in [0, 2]$. Then
	\begin{align*}
		\lim _{r \to 0^+}
		\psi\big(r^{-1} \big)
		G\big(\{x : \norm{x} \leq r\}\big)
		= 
		\mathcal{\tilde{C}}_{d, \alpha}.
	\end{align*}
\end{theorem}

\begin{corollary}
	\label{cor:2}
	Assume that $\psi \in \RInf$, for some $\alpha \in (0, 2]$. Then
	\begin{align*}
		\lim _{x \to 0} \norm{x}^{d} \psi\big(\norm{x}^{-1}\big) G(x)= \mathcal{\tilde{A}}_{d,\alpha}.
	\end{align*}
\end{corollary}

\subsection{Examples}
We begin with a result which is helpful in verifying whether the L\'{e}vy--Khintchine exponent $\psi $ is regularly
varying.
\begin{proposition}
	Let $\mathbf X$ be a pure-jump isotropic unimodal L\'{e}vy processes with the L\'{e}vy density
	$\nu(r) = r^{-d} g\big(r^{-1}\big)$, for some $g:(0, \infty) \to [0, \infty)$. 
	\begin{enumerate}
		\item If $g \in \ROrg$ for some $\alpha\in(0,2)$, then $\psi\in \ROrg$ and 
		\[
			\lim_{r \to 0^+} \frac{g(r)}{\psi(r)}=\mathcal{A}_{d,\alpha}.
		\]
	\item If $g\in \RInf$  for some $\alpha\in(0,2)$, then $\psi\in \RInf$ and 
		\[
			\lim_{r\to +\infty}\frac{g(r)}{\psi(r)}=\mathcal{A}_{d,\alpha}.
		\]
	\end{enumerate}
\end{proposition}
\begin{proof}
	For $r > 0$, we have (see \eqref{eq:40})
	\[
		\psi(r)
		= \int^\infty_0 k(r \rho) g(\rho^{-1}) \frac{{\rm d}\rho}{\rho}
		= \int^\infty_0 k( \rho) g(r\rho^{-1})\frac{{\rm d}\rho}{\rho}.
	\]
	Assume that $g\in \ROrg$. Let $0 < \varepsilon < 2-\alpha$ then, by \eqref{eq:14}, there is $0< \delta \leq 1$
	such that for $r,r\rho^{-1} \leq \delta$
	\[
		g\big(r \rho^{-1}\big) \leq 2 g(r) \rho^{-\alpha -\varepsilon}.
	\]
	Moreover, by \eqref{eq:50}
	\[
		\int^{r \delta^{-1}}_0 k( \rho) g\big(r\rho^{-1}\big) \frac{{\rm d}\rho}{\rho}
		= \int^{\delta^{-1}}_0 k( r\rho) g\big(\rho^{-1}\big)\frac{{\rm d}\rho}{\rho}
		\lesssim r^2.
	\]
	Hence, the first claim of the proposition is a consequence of the dominated convergence theorem and
	a fact that
	\[
		\mathcal{A}_{d,\alpha}
		=\int^\infty_0 k(\rho) \frac{{\rm d}\rho}{\rho^{1+\alpha}}.
	\]
	In the similar way one can prove the second claim.
\end{proof}
\begin{example}
	Let $\mathbf X$ be the relativistic stable process, i.e. $\psi(x)=(\norm{x}^2+1)^{\alpha/2}-1$, $\alpha\in(0,2)$.
	We have
	\begin{equation}
		\label{eq:41}
		\lim_{\atop{x \to 0}{t \norm{x}^{-\alpha} \to 0}}
		\frac{p(t, x)}{t \norm{x}^{-d - \alpha}} 
		= \mathcal{A}_{d, \alpha}.
	\end{equation}
\end{example}
\begin{example}
	Let $\nu(r)=r^{-d} g(r)$, $r\in(0,\infty)$ and $\alpha,\alpha_1\in(0,2)$. The same limit as \eqref{eq:41} exists and
	equals 1 for
	\begin{enumerate}
		\item truncated stable process: $g(r)=r^{-\alpha} \ind{(0, 1)}(r)$;
		\item tempered stable process: $g(r)=r^{-\alpha} e^{-r}$;
		\item isotropic Lamperti stable process: $g(r)=re^{\delta r}(e^r-1)^{-\alpha-1}$, $\delta<\alpha+1$;
		\item layered stable process: $g(r)=r^{-\alpha} \ind{(0,1)(r)} + r^{-\alpha_1} \ind{[1, \infty)}(r)$.
	\end{enumerate}
Since distributions of  processes (i)--(iii) have a finite second moment the dominated convergence theorem yields
\[
	\lim_{r\to 0^+} r^{-2} \psi(r)
	=
	2^{-1} \int^\infty_0 \rho g(\rho) {\: \rm d} \rho=c_1,
\]
what implies, for $d\geq 3$, 
\[
	\lim_{\norm{x} \to +\infty} \norm{x}^{2-d} G(x)
	=
	c_1 \mathcal{\tilde{A}}_{d,2}.
\]
\end{example}

\begin{example}
	Let 
	\[
		\psi(\xi)=\norm{\xi}^{\alpha} \log^{\beta}(1+\norm{\xi}^{\gamma}),
	\]
	where $\gamma,\alpha,\alpha+2\beta\in (0,2)$.  We note that $\psi$ is the L\'{e}vy--Khintchine exponent of
	a subordinate Brownian motion, see \cite[Theorem 12.14, Proposition 7.10, Proposition 7.1, Corollary 7.9,
	Section 13 and examples 1 and 26 from Section 15.2]{ssv}. For this process we have
	\[
		\lim _{\atop{x \to 0}{t \norm{x}^{-\alpha} \log^\beta{\norm{x}} \to 0}}
		\frac{p(t, x)}{t \norm{x}^{-d - \alpha} \log^\beta \norm{x}} 
		= \gamma^{\beta}\mathcal{ A}_{d, \alpha},
	\]
	and
	\[
		\lim _{\atop{\norm{x} \to +\infty} {t \norm{x}^{-\alpha-\gamma\beta} \to 0}}
		\frac{p(t, x)}{t \norm{x}^{-d - \alpha-\gamma\beta}} =\mathcal{ A}_{d, \alpha+\gamma\beta}.
	\]
\end{example}

\begin{example}
	Let $\alpha\in(0,2)$ and $\beta\in\RR$ and $\mathbf X$ be isotropic unimodal with the L\'{e}vy density
	$\nu(r)=r^{-d-\alpha} \log^\beta(1+e^\beta+r)$. Then
	\[
		\lim _{\atop{\norm{x} \to \infty}{t \norm{x}^{-\alpha} \log^\beta{\norm{x}} \to 0}}
		\frac{p(t, x)}{t \norm{x}^{-d - \alpha} \log^\beta \norm{x}} 
		= 1.
	\]

\end{example}
\begin{example}
	Let $\mathbf X$ be the gamma variance process, i.e. $\psi(\xi)=\log(1+\norm{\xi}^2)$. Then
	\[
		\lim_{\atop{r \to 0^+}{t \abs{\log r} \to 0}}
		\frac{\PP(\norm{X_t} > r)}{t \abs{\log r}} 
		= 2.
	\]
\end{example}

\providecommand{\bysame}{\leavevmode\hbox to3em{\hrulefill}\thinspace}
\providecommand{\MR}{\relax\ifhmode\unskip\space\fi MR }
\providecommand{\MRhref}[2]{%
  \href{http://www.ams.org/mathscinet-getitem?mr=#1}{#2}
}
\providecommand{\href}[2]{#2}

\end{document}